\documentclass{publmathdeb}
\usepackage{amsmath,amsfonts,amssymb, enumerate}

\theoremstyle{plain}

\newtheorem{thm}{Theorem}[section]

\theoremstyle{definition}

\theoremstyle{remark}
\newtheorem{remark}{Remark}[section]

\newtheorem{Cor}{Corollary}[section]
\theoremstyle{remark}

\numberwithin{equation}{section}

\begin{document}

\title[Stability of the entropy equation]
{Stability of the  entropy equation }
\author{Eszter Gselmann}
\address{Institute of Mathematics\\
University of Debrecen\\
H--4010 Debrecen, P.~O.~Box 12.}
\email{gselmann@math.klte.hu}



\submitted{\today}

\thanks{This research has been supported by the Hungarian Scientific Research Fund
(OTKA) Grant NK 68040 and also by the Universitas Scholarship founded by the K\&H Bank
(Kereskedelmi \'{e}s Hitelbank Rt.).}

\begin{abstract}
In this paper we prove that the so--called entropy equation, i.e.,
\[
H\left(x, y, z\right)=H\left(x+y, 0, z\right)+H\left(x, y, 0\right)
\]
is stable in the sense of Hyers and Ulam on the
positive cone of $\mathbb{R}^{3}$, assuming that the function $H$
is approximatively symmetric in each variable and
approximatively homogeneous of degree $\alpha$, where
$\alpha$ is an arbitrarily fixed real number.
\end{abstract}

\maketitle

\section{Introduction and preliminaries}

The stability theory of functional equation originates from a famous question of
S.~M. Ulam concerning the additive Cauchy equation. He asked whether is it true that
the solution of the additive Cauchy equation differing slightly
from a given one, must of necessity be close to the solution of this
equation. In 1941 D.~H.~Hyers gave an affirmative answer to the previous
question.
Nowadays this result (see \cite{Hye41}) is referred to as
\emph{the stability of the Cauchy equation}.
Since then the stability theory of functional equations
has become a developing field of research, see e.g.,
\cite{For95}, \cite{Ger94}, \cite{Mos04}
and their references.

In the theory of stability there exist several methods.
From our point of view however the method of invariant means plays a key role.
Concerning this topic we offer the expository paper Day \cite{Day57}.
Although the only result needed from \cite{Day57} is,
that on every commutative semigroup there
exist an invariant mean, that is,
every commutative semigroup is \emph{amenable}.

In what follows, denote $\mathbb{R}$ the set of the real numbers,
furthermore, on the symbols $\mathbb{R}_{+}$ and $\mathbb{R}_{++}$ we
understand the set of the nonnegative and the positive
real numbers, respectively.

The aim of this paper is to prove that \emph{the entropy equation}, i.e., equation
\begin{equation}\label{ent.eq}
H\left(x, y, z\right)=H\left(x+y, 0, z\right)+H\left(x, y, 0\right)
\end{equation}
is stable on $\mathbb{R}^{3}_{++}$.

In \cite{KM74} A.~Kami\'{n}ski and J.~Mikusi\'{n}ski
determined the continuous and $1$--homogeneous solutions of
equation (\ref{ent.eq}).
This result was strengthened by
J.~Acz\'{e}l in \cite{Acz77}, by proving the following.

\begin{thm}
Let
\[
D=\left\{(x, y, z)\in\mathbb{R}^{3}\vert x\geq 0, y\geq 0, z\geq 0, x+y+z>0\right\}.
\]
Assume that the function $H:D\rightarrow\mathbb{R}$ is symmetric, positively homogeneous of
degree $1$, satisfies the functional equation
\[
H\left(x, y, z\right)=H\left(x+y, z, 0\right)+H\left(x, y, 0\right)
\]
on $D$ and the map $x\longmapsto H\left(1-x, x, 0\right)$ is either continuous at a point
or bounded on an interval or integrable on
the closed subintervals of $]0,1[$ or measurable on $]0,1[$. Then, and only then
\[
H\left(x, y, z\right)=
x\log\left(x\right)+y\log\left(y\right)+z\log\left(z\right)-
\left(x+y+z\right)\log\left(x+y+z\right)
\]
holds for all $\left(x, y, z\right)\in D$, with arbitrary
basis for the logarithm and with the convention $0\cdot \log\left(0\right)=0$.
\end{thm}

Using a result of Jessen--Karpf--Thorup \cite{JKT69}, which concerns the solution of
the cocycle equation, Z.~Dar\'{o}czy proved the following (see \cite{Dar76}).

\begin{thm}
If a function $H:D\rightarrow\mathbb{R}$ is symmetric in $D$ and
satisfies the equation (\ref{ent.eq}) in the interior of $D$ and
the map $(x, y)\mapsto H\left(x, y, 0\right)$ is positively
homogeneous (of order $1$) for all $x, y\in\mathbb{R}_{++}$,
then there exists a function
$\varphi:\mathbb{R}_{++}\rightarrow\mathbb{R}$ such that
\[
\varphi\left(xy\right)=x\varphi\left(y\right)+y\varphi\left(x\right)
\]
holds for all $x, y \in\mathbb{R}_{++}$ and
\[
H\left(x, y, z\right)=\varphi\left(x+y+z\right)-\varphi(x)-\varphi(y)-\varphi(z)
\]
for all $\left(x, y, z\right)\in D$.
\end{thm}

During the proof of the main result the stability of the \emph{cocycle equation} is needed.
This theorem can be found in \cite{Sze95}.

\begin{thm}\label{szekely}
Let $S$ be a right amenable semigroup and let
$F:S\times S\rightarrow\mathbb{C}$ be a function, for which the function
\begin{equation}\label{Eqc}
\left(x, y, z\right)\longmapsto
F\left(x, y\right)+F\left(x+y, z\right)-F\left(x, y+z\right)-F\left(y, z\right)
\end{equation}
is bounded on $S\times S\times S$.
Then there exists a function
$\Psi:S\times S\rightarrow\mathbb{C}$ satisfying the cocycle equation,
i.e.,
\begin{equation}\label{cocycle}
\Psi\left(x, y\right)+\Psi\left(x+y, z\right)=\Psi\left(x, y+z\right)+\Psi\left(y, z\right)
\end{equation}
for all $x, y, z \in S$ and for which the function $F-\Psi$ is bounded by the same constant as
the map defined by (\ref{Eqc}).
\end{thm}

About the symmetric, $1$--homogeneous solutions of the cocycle equation  one
can read in \cite{JKT69}.
Furthermore, the symmetric and $\alpha$--homogeneous solutions of equation
(\ref{cocycle}) can be found in \cite{Mak82}, as a consequence of Theorem 3.
The general solution of the cocycle equation without symmetry and homogeneity assumptions,
on cancellative abelian semigroups
was determined by M.~Hossz\'{u} in \cite{Hos71}.

\section{The main result}

Our main result is the following.
\begin{thm}\label{main}
Let $\varepsilon_{1}, \varepsilon_{2}, \varepsilon_{3}$ be arbitrary
nonnegative real numbers, $\alpha\in\mathbb{R}$, and assume that the function
$H:D\rightarrow\mathbb{R}$ satisfies the following system of inequalities.
\begin{equation}\label{Eq1}
\left|H(x, y, z)-H\left(\sigma(x), \sigma(y), \sigma(z)\right)\right|\leq \varepsilon_{1}
\end{equation}
for all $(x, y, z)\in D$ and for all
$\sigma:\left\{x, y, z\right\}\mapsto\left\{x, y, z\right\}$
permutation;
\begin{equation}\label{Eq2}
\left|H\left(x, y, z\right)-H\left(x+y, 0, z\right)-H\left(x, y, 0\right)\right|\leq \varepsilon_{2}
\end{equation}
for all $(x, y, z)\in D^{\circ}$, where $D^{\circ}$ denotes the interior of the set $D$;
\begin{equation}\label{Eq3}
\left|H\left(tx, ty, 0\right)-t^{\alpha}H(x, y, 0)\right|\leq \varepsilon_{3}
\end{equation}
holds for all $t, x, y\in\mathbb{R}_{++}$.
Then, in case $\alpha=1$ there exists a function $\varphi:\mathbb{R}_{++}\rightarrow\mathbb{R}$ which satisfies
the functional equation
\[
\varphi\left(xy\right)=x\varphi\left(y\right)+y\varphi\left(x\right),
\quad \left(x, y\in\mathbb{R}_{++}\right)
\]
and
\begin{equation}\label{Eq4}
\left|H\left(x, y, z\right)-
\left[\varphi\left(x+y+z\right)-\varphi\left(x\right)-\varphi\left(y\right)-\varphi\left(z\right)\right]\right|
\leq \varepsilon_{1}+\varepsilon_{2}
\end{equation}
holds for all $\left(x, y, z\right)\in D^{\circ}$;
in case $\alpha=0$ there exists a constant $a\in\mathbb{R}$ such that
\begin{equation}\label{Eq5}
\left|H\left(x, y, z\right)-a\right|\leq
8\varepsilon_{3}+25\varepsilon_{2}+49\varepsilon_{1}
\end{equation}
for all $\left(x, y, z\right)\in D^{\circ}$;
finally, in all other cases there exists a constant $c\in\mathbb{R}$ such that
\begin{equation}\label{Eq6}
\left|H\left(x, y, z\right)-c\left[\left(x+y+z\right)^{\alpha}-x^{\alpha}-y^{\alpha}-z^{\alpha}\right]\right|
\leq \varepsilon_{1}+\varepsilon_{2}
\end{equation}
holds on $D^{\circ}$.
\end{thm}

\begin{proof}
Using inequality (\ref{Eq3}) we will show that the map
$(x, y)\mapsto H\left(x, y, 0\right)$ is homogeneous
of degree $\alpha$, in fact, assuming that $\alpha\neq 0$.
Due to (\ref{Eq3})
\[
\left|H(x, y, 0)-\frac{H(tx, ty, 0)}{t^{\alpha}}\right|\leq \frac{\varepsilon_{3}}{t^{\alpha}}
\]
holds for all $t, x, y\in \mathbb{R}_{++}$.
Hence, if we define
\[
t_{0}=\left\{
\begin{array}{lcl}
0 & \text{, if}& \alpha<0 \\
+\infty  & \text{, if} & \alpha>0,
\end{array}
\right.
\]
then for all $x, y\in\mathbb{R}_{++}$
\begin{equation}
\lim_{t\rightarrow t_{0}} t^{-\alpha}H\left(tx, ty, 0\right)= H\left(x, y, 0\right).
\end{equation}
Thus we have for arbitrary $s, x, y\in\mathbb{R}_{++}$
\begin{equation}\label{Eq7}
\begin{array}{lcl}
H(sx, sy, 0)
 &=&\lim_{t\rightarrow t_{0}}t^{-\alpha}H(tsx, tsy, 0) \\
 &=&\lim_{t\rightarrow t_{0}}t^{-\alpha}s^{-\alpha}H((ts)x, (ts)y, 0)s^{\alpha}\\
 &=&s^{\alpha}\lim_{t\rightarrow t_{0}}(ts)^{-\alpha}H((ts)x, (ts)y, 0)\\
 &=&s^{\alpha}H(x, y, 0).
\end{array}
\end{equation}
Therefore, the map $(x, y)\mapsto H\left(x, y, 0\right)$
 is homogeneous of degree $\alpha$, indeed.

In what follows, we will investigate inequalities (\ref{Eq1}) and (\ref{Eq2}).
Interchanging $x$ and $z$ in (\ref{Eq2}), we obtain that
\begin{equation}\label{Eq8}
\left|H\left(z, y, x\right)-H\left(z+y, 0, x\right)-H\left(z, y, 0\right)\right|\leq \varepsilon_{2}
\end{equation}
is satisfied for all $(x, y, z)\in D^{\circ}$.
Inequalities (\ref{Eq1}), (\ref{Eq2}), (\ref{Eq8}) and the triangle inequality imply that
\begin{equation}\label{Eq9}
\left|H\left(x+y, 0, z\right)+H\left(x, y, 0\right)-
H\left(y+z, 0, x\right)-
H\left(z, y, 0\right)\right|\leq 2\varepsilon_{2}+\varepsilon_{1}
\end{equation}
is fulfilled for all $(x, y, z)\in D^{\circ}$.
Applying inequality (\ref{Eq1}) three times, from (\ref{Eq9}), we get that
\begin{equation}\label{Eq10}
\left|H\left(x+y, z, 0\right)+
H\left(x, y, 0\right)-
H\left(x, y+z, 0\right)-
H\left(y, z, 0\right)\right|\leq 2\varepsilon_{2}+4\varepsilon_{1}
\end{equation}
holds on $D^{\circ}$.
Therefore, the function $F$ defined by
\begin{equation}\label{Fdef}
F(x, y)=H(x, y, 0) \quad (x, y\in\mathbb{R}_{++})
\end{equation}
satisfies
\begin{equation}\label{Eq11}
\left|F(x, y)-F(y, x)\right|\leq \varepsilon_{1}, \quad \left(x, y\in \mathbb{R}_{++}\right)
\end{equation}
due to inequality (\ref{Eq1}). Since, for the function $H$
inequality (\ref{Eq10}) holds, we receive that
\begin{equation}\label{Eq12}
\left|F(x+y, z)+
F(x, y)-
F(x, y+z)-
F(y, z)\right|\leq 2\varepsilon_{2}+4\varepsilon_{1}.
\quad (x, y, z\in\mathbb{R}_{++})
\end{equation}
We have just proved that in case $\alpha\neq 0$, $H(x, y, 0)$ is
homogeneous of degree $\alpha$, therefore
\begin{equation}\label{Eq13}
F(tx, ty)=t^{\alpha}F(x, y) \quad \left(\alpha\neq 0, t, x, y \in\mathbb{R}_{++}\right)
\end{equation}
furthermore, because of (\ref{Fdef}) and (\ref{Eq3}), we obtain that
\begin{equation}\label{Eq14}
\left|F(tx, ty)-F(x, y)\right|\leq \varepsilon_{3},  \quad \left(t, x, y\in\mathbb{R}_{++}\right)
\end{equation}
in case $\alpha=0$.

The set $D^{\circ}$ is a commutative semigroup with the usual addition.
Thus it is amenable, as well.
Therefore, by Theorem \ref{szekely}., there exists a function
$G:\mathbb{R}^{2}_{++}\rightarrow\mathbb{R}$ which is a solution of the cocycle
equation, and for which
\begin{equation}\label{EQ20}
\left|F\left(x, y\right)-G\left(x, y\right)\right|
\leq 2\varepsilon_{2}+4\varepsilon_{1}
\end{equation}
holds for all $x, y\in\mathbb{R}_{++}$.
Additionally, by a result of
\cite{Hos71} there exist a function $f:\mathbb{R}_{++}\rightarrow\mathbb{R}$ and a function
$B:\mathbb{R}^{2}_{++}\rightarrow\mathbb{R}$ which satisfies the following system
\[
\begin{array}{rcl}
B(x+y, z)&=&B(x, z)+B(y, z), \\
B(x, y)+B(y, x)&=&0,
\end{array}
\quad (x, y, z\in\mathbb{R}_{++})
\]
such that
\[
G\left(x, y\right)=B\left(x, y\right)+f\left(x+y\right)-f\left(x\right)-f\left(y\right).
\quad \left(x, y\in\mathbb{R}_{++}\right)
\]
All in all, this means that
\begin{equation}\label{Eq15}
\left|F(x, y)-\left(B(x, y)+f(x+y)-f(x)-f(y)\right)\right|\leq 2\varepsilon_{2}+4\varepsilon_{1}
\end{equation}
holds for all $x, y\in\mathbb{R}_{++}$.

Using the above properties of the function $B$, we
will show that $B$ is identically zero on $\mathbb{R}^{2}_{++}$.
Indeed, by reason of the triangle inequality and (\ref{Eq15}),
\[
\begin{array}{l}
\left|2B(x, y)\right|=\left|B(x, y)-B(y, x)\right| \\
\leq
\left|F(x, y)-\left(B(x, y)+f(x+y)-f(x)-f(y)\right)\right|\\
+
\left|F(y, x)-\left(B(y, x)+f(y+x)-f(y)-f(x)\right)\right|\\
+ \left|F(x, y)-F(y, x)\right| \\
\leq (2\varepsilon_{2}+4\varepsilon_{1})+(2\varepsilon_{2}+4\varepsilon_{1})+\varepsilon_{1}=
4\varepsilon_{2}+9\varepsilon_{1}
\end{array}
\]
is fulfilled for all $x, y\in\mathbb{R}_{++}$. Thus $B$ is bounded on the set $\mathbb{R}^{2}_{++}$.
On the other hand, $B$ is biadditive. However, only the identically zero function has these
properties. Therefore, $B\equiv 0$ on $\mathbb{R}_{++}$.

Then we get that the function $K:\mathbb{R}^{2}_{++}\rightarrow\mathbb{R}$ defined by
\[
K(x, y)=F(x, y)-G(x, y) \quad \left(x, y\in\mathbb{R}_{++}\right)
\]
is bounded on $\mathbb{R}^{2}_{++}$ by $2\varepsilon_{2}+4\varepsilon_{1}$,
in view of inequality (\ref{Eq15}).
In case $\alpha\neq 0$, it can be seen that
\[
K\left(tx, ty\right)=F\left(tx, ty\right)-G\left(tx, ty\right),  \quad \left(t, x, y\in\mathbb{R}_{++}\right)
\]
furthermore, making use of (\ref{Eq13})
\[
K\left(tx, ty\right)=t^{\alpha}F\left(x, y\right)-G\left(tx, ty\right)
\]
holds for all $t, x, y\in\mathbb{R}_{++}$.
Rearranging this,
\[
\frac{K\left(tx, ty\right)}{t^{\alpha}}=F(x, y)-\frac{G\left(tx, ty\right)}{t^{\alpha}}
\]
for all $t, x, y\in\mathbb{R}_{++}$.
Since the function $K$ is bounded on $\mathbb{R}^{2}_{++}$, we receive that
\[
F(x, y)=\lim_{t\rightarrow t_{0}}\frac{G\left(tx, ty\right)}{t^{\alpha}}.   \quad \left(x, y\in\mathbb{R}_{++}\right)
\]

Because of the symmetry of the function $G$, the function $F$ is symmetric, as well.
Furthermore, $G$ satisfies the cocycle equation on $\mathbb{R}_{++}$, that is,
\[
G(x+y, z)+G(x, y)=G(x, y+z)+G(y, z),  \quad \left(x, y, z\in\mathbb{R}_{++}\right)
\]
especially,
\[
\frac{G(tx+ty, tz)}{t^{\alpha}}+\frac{G(tx, ty)}{t^{\alpha}}=
\frac{G(tx, ty+tz)}{t^{\alpha}}+\frac{G(ty, tz)}{t^{\alpha}}
\]
is also satisfied for all $t, x, y, z\in\mathbb{R}_{++}$.
Taking the limit $t\rightarrow t_{0}$ we obtain that
\[
F(x+y, z)+F(x, y)=F(x, y+z)+F(y, z). \quad \left(x, y, z\in\mathbb{R}_{++}\right)
\]
This means that also the function $F$ satisfies the cocycle equation on $\mathbb{R}^{2}_{++}$.
Additionally, $F$ is homogeneous of degree $\alpha$ ($\alpha\neq 0$) and symmetric.
Using Theorem 5. in \cite{JKT69}, in case
$\alpha=1$, and a result of \cite{Mak82} in all other cases,
we get that
\begin{equation}
F(x, y)=\left\{
\begin{array}{lcl}
c\left[(x+y)^{\alpha}-x^{\alpha}-y^{\alpha}\right], & \hbox{if} & \alpha\notin \left\{0,1\right\} \\
\varphi\left(x+y\right)-\varphi(x)-\varphi(y), & \hbox{if}& \alpha=1
\end{array}
\right.
\end{equation}
where the function $\varphi:\mathbb{R}_{++}\rightarrow\mathbb{R}$ satisfies the functional equation
\[
\varphi\left(xy\right)=x\varphi(y)+y\varphi(x)
\]
for all $x, y\in\mathbb{R}_{++}$, and $c\in\mathbb{R}$ is a constant.
In view of the definition of the function $F$, this yields that
\begin{equation}\label{EQ23}
H(x, y, 0)= c\left[(x+y)^{\alpha}-x^{\alpha}-y^{\alpha}\right]
\end{equation}
for all $x, y\in\mathbb{R}_{++}$ in case $\alpha\notin\left\{0, 1\right\}$, and
\begin{equation}\label{EQ24}
H(x, y, 0)=\varphi\left(x+y\right)-\varphi(x)-\varphi(y)
\end{equation}
for all $x, y\in\mathbb{R}_{++}$ in case $\alpha=1$.
Finally, inequalities (\ref{Eq1}), (\ref{Eq2}) and equation (\ref{EQ23}) imply that
\begin{equation}
\begin{array}{l}
\left|H(x, y, z)-c\left[(x+y+z)^{\alpha}-x^{\alpha}-y^{\alpha}-z^{\alpha}\right]\right| \\
\leq
\left|H(x, y, z)-H(x+y, 0, z)-H(x, y, 0)\right| \\
+\left|H(x+y, 0, z)-H(x+y, z, 0)\right|\\
+\left|H(x+y, z, 0)-c\left[(x+y+z)^{\alpha}-(x+y)^{\alpha}-z^{\alpha}\right]\right| \\
+\left|H(x, y, 0)-c\left[(x+y)^{\alpha}-x^{\alpha}-y^{\alpha}\right]\right| \\
\leq \varepsilon_{1}+\varepsilon_{2}
\end{array}
\end{equation}
for all $x, y, z\in\mathbb{R}_{++}$, if $\alpha\notin\left\{0, 1\right\}$, and by reason of
inequalities (\ref{Eq1}), (\ref{Eq2}) and equation (\ref{EQ24}) we obtain that
\begin{equation}
\begin{array}{l}
\left|H(x, y, z)-\left(\varphi\left(x+y+z\right)-\varphi(x)-\varphi(y)-\varphi(z)\right)\right| \\
\leq
\left|H(x, y, z)-H(x+y, 0, z)-H(x, y, 0)\right| \\
+\left|H(x+y, 0, z)-H(x+y, z, 0)\right|\\
+\left|H(x+y, z, 0)-\left(\varphi(x+y+z)-\varphi(x+y)-\varphi(z)\right)\right| \\
+\left|H(x, y, 0)-\left(\varphi(x+y)-\varphi(x)-\varphi(y)\right)\right| \\
\leq \varepsilon_{1}+\varepsilon_{2}
\end{array}
\end{equation}
for all $x, y, z\in\mathbb{R}_{++}$, if $\alpha=1$.

In case $\alpha=0$, we get from (\ref{Eq14}), that particularly
\[
\left|F(x, y)-F(2x, 2y)\right|\leq \varepsilon_{3}.  \quad \left(x, y\in\mathbb{R}_{++}\right)
\]
Thus inequality (\ref{EQ20}) implies
\[
\begin{array}{l}
\left|F(x, y)-G\left(2x, 2y\right)\right|\\
\leq\left|F\left(x, y\right)-F\left(2x, 2y\right)\right|
+\left|F\left(2x, 2y\right)-G\left(2x, 2y\right)\right|\\
\leq \varepsilon_{3}+2\varepsilon_{2}+4\varepsilon_{1}
\end{array}
\]
holds for all $x, y\in\mathbb{R}_{++}$.
On the other hand
\[
\left|G\left(2x, 2y\right)-\left[2G\left(x, y\right)-F\left(1, 1\right)\right]\right|
\leq 3\left(\varepsilon_{3}+2\varepsilon_{2}+4\varepsilon_{1}\right)
\]
for all $x, y\in\mathbb{R}_{++}$, since,
\[
\begin{array}{l}
\left|F\left(1, 1\right)-G\left(x, x\right)\right|\\
\leq
\left|F\left(1, 1\right)-F\left(x, x\right)\right|+
\left|F\left(x, x\right)-G\left(x, x\right)\right|\\
 \leq
\varepsilon_{3}+2\varepsilon_{2}+4\varepsilon_{1}
\end{array}
\]
where we used inequalities (\ref{Eq14}) and (\ref{Eq15}).
All in all, this means that
\begin{equation}\label{EQ27}
\begin{array}{l}
\left|F\left(x, y\right)-F\left(1, 1\right)\right| \\
\leq \left|G\left(2x, 2y\right)-F\left(x, y\right)\right| \\
+\left|G\left(2x, 2y\right)-2G\left(x, y\right)+F\left(1, 1\right)\right| \\
+2\left|G\left(x,y\right)-F\left(x, y\right)\right| \\
\leq \varepsilon_{3}+2\varepsilon_{2}+4\varepsilon_{1}
+3\left(\varepsilon_{3}+2\varepsilon_{2}+4\varepsilon_{1}\right)
+2\left(2\varepsilon_{2}+4\varepsilon_{1}\right)\\=
4\varepsilon_{3}+12\varepsilon_{2}+24\varepsilon_{1}
\end{array}
\end{equation}
is fulfilled for all $x, y\in\mathbb{R}_{++}$.

Due to the definition of the function $F$,
\begin{equation}\label{EQ28}
\left|H(x, y, 0)-F(1, 1)\right|\leq
4\varepsilon_{3}+12\varepsilon_{2}+24\varepsilon_{1}
\end{equation}
for all $x, y\in\mathbb{R}_{++}$, where we used inequality
\eqref{EQ27}.
Finally, in view of (\ref{Eq1}), (\ref{Eq2})
and (\ref{EQ28}) we obtain that
\begin{equation}
\begin{array}{l}
\left|H(x, y, z)-2F(1, 1)\right|\\ \leq
\left|H(x, y, z)-H(x+y, 0, z)-H(x, y, 0)\right|\\
+ \left|H(x+y, 0, z)-H(x+y, z, 0)\right| \\
+\left|H(x+y, z, 0)-F(1, 1)\right|+\left|H(x, y, 0)-F(1, 1)\right| \\
\leq \varepsilon_{2}+\varepsilon_{1}+
(4\varepsilon_{3}+12\varepsilon_{2}+24\varepsilon_{1})
+(4\varepsilon_{3}+12\varepsilon_{2}+24\varepsilon_{1})\\
=8\varepsilon_{3}+25\varepsilon_{2}+49\varepsilon_{1}
\end{array}
\end{equation}
holds for all $(x, y, z)\in D^{\circ}$. Let $a=2F(1, 1)$ to get the desired inequality.
\end{proof}

With the choice $\varepsilon_{1}=\varepsilon_{2}=\varepsilon_{3}=0$ one can
recognize the solutions
of equation (\ref{ent.eq}).

\begin{Cor}
Assume that the function
$H:D\rightarrow\mathbb{R}$ is symmetric, homogeneous of degree $\alpha$, where
$\alpha\in\mathbb{R}$ is arbitrary but fixed. Furthermore, suppose that
$H$ satisfies equation (\ref{ent.eq}) on the set $D^{\circ}$.
Then, in case $\alpha=1$ there exists a function $\varphi:\mathbb{R}_{++}\rightarrow\mathbb{R}$ which satisfies
the functional equation
\[
\varphi\left(xy\right)=x\varphi\left(y\right)+y\varphi\left(x\right),
\quad \left(x, y\in\mathbb{R}_{++}\right)
\]
and
\begin{equation}
H\left(x, y, z\right)=
\varphi\left(x+y+z\right)-\varphi\left(x\right)-\varphi\left(y\right)-\varphi\left(z\right)
\end{equation}
holds for all $\left(x, y, z\right)\in D^{\circ}$;
in all other cases there exists a constant $c\in\mathbb{R}$ such that
\begin{equation}
H\left(x, y, z\right)=c\left[\left(x+y+z\right)^{\alpha}-x^{\alpha}-y^{\alpha}-z^{\alpha}\right]
\end{equation}
holds on $D^{\circ}$.
\end{Cor}

\begin{remark}
Our theorem says that the entropy equation is stable in the sense of Hyers and Ulam.
\end{remark}

\begin{remark}
In 2005 an article of J.~Tabor and J.~Tabor has appeared with
exactly the same title as that of the present paper.
However, in \cite{TT05} the Hyers--Ulam stability of the functional equation
\[
L\left(\sum^{3}_{j=1}k_{j}f\left(p_{j}\right)\right)=
\sum^{3}_{j=1}k_{j}g\left(p_{j}\right)
\]
is proved, for the Banach space $X$,
$0\leq p_{j}\leq 1$, $k_{j}\in\mathbb{N}\cup \left\{0\right\}$,
$\sum^{3}_{j=1}k_{j}p_{j}=1$, where $f:[0,1]\rightarrow\mathbb{R}_{+}$,
$g:[0,1]\rightarrow X$ and
$L:\mathbb{R}_{+}\rightarrow X$ unknown continuous functions satisfying
some additional conditions.
\end{remark}

\textbf{Acknowledgement.}
The author is grateful to Professor Gyula Maksa for his helpful comments
and permanent encouragement during the preparation of the manuscript.

\end{document}